\documentclass[twoside,12pt,reqno]{amsart}

\usepackage{upref,amsxtra,amssymb,amscd}
\usepackage{varioref}
\usepackage{verbatim}
\usepackage{epsfig}
\usepackage{color}
  
\usepackage{epic,eepic,eucal}
\usepackage{enumerate}

\def\Or{          \mathcal O}
\def\Rc{          \mathcal R}

\def\E{          \mathcal E}
\def\D{          \mathcal D}

\def\a{         \alpha}
\def\b{         \beta}

\def\T{          \mathcal T}

\def\presup#1{{}^{#1}\kern-.10em\relax} 
\def\presub#1{{}_{#1}\kern-.12em\relax}       
      
\newcommand{\NN}{{\mathbb N}}
\newcommand{\RR}{{\mathbb R}}
\newcommand{\TT}{{\mathbb T}}

\newcommand{\ZZ}{{\mathbb Z}}

\newtheorem{theo}{\sc Theorem}[section]

\newtheorem{lemma}[theo]{\sc Lemma}
\newtheorem{sublemm}[theo]{\sc Sublemma}
\newtheorem{lemm}[theo]{\sc Lemma}
\newtheorem{coro}[theo]{\sc Corollary}

\newtheorem*{property}{\sc Property}

\theoremstyle{definition}

\theoremstyle{remark}

\newtheorem*{rema}{\bf Remark}

\newtheorem{ques}{\bf Question}

\numberwithin{equation}{section}

\begin{document}
\title{Schmidt games and Markov partitions}
\author{Jimmy Tseng}
\address{Jimmy Tseng, Department of Mathematics, Brandeis University, Waltham, MA 02454} 
\email{jtseng@brandeis.edu}

\begin{abstract}Let $T$ be a $C^2$-expanding self-map of a compact, connected, $C^\infty$, Riemannian manifold $M$.  We correct a minor gap in the proof of a theorem from the literature:  the set of points whose forward orbits are nondense has full Hausdorff dimension.  Our correction allows us to strengthen the theorem. 

Combining the correction with Schmidt games, we generalize the theorem in dimension one:  given a point $x_0 \in M$, the set of points whose forward orbit closures miss $x_0$ is a winning set.
\end{abstract}


\maketitle

\section{Introduction}

Let $T:M\rightarrow M$ be a $C^2$-expanding self-map of a compact, connected, $C^\infty$, Riemannian manifold $M$ with volume measure $\sigma$.  In this note, we study the set of points whose forward orbits are nondense.  There is an ergodic $T$-invariant probability measure equivalent to $\sigma$~\cite{KS}.  Consequently, this set has zero volume by the Birkhoff ergodic theorem, but it also has full Hausdorff dimension (equal to $\dim M$)~\cite{Ur}.\footnote{In addition to~\cite{Ur}, see Subsections~\ref{subsecOOP} and~\ref{subsecCorrect} below.}  As a result, this set is small in terms of measure, but large in terms of Hausdorff dimension.  In particular, it is uncountable.  

There are a number of similar theorems where one investigates the character of nondense orbits for systems with some hyperbolic behavior.  For example, there is a theorem for the homogeneous space $H:=SL_2(\RR)/SL_2(\ZZ)$ which uses the classical result that the set of badly approximable numbers\footnote{Recall that a number $\a \in \RR$ is \textbf{badly approximable} if there exists a constant $C(\a) >0$ such that \[|\a - p/q| > \frac{C(\a)}{q^2}\] for all $p, q \in \ZZ$ and $q \neq 0$.} has full Hausdorff dimension.  There is a bijection between a real number $\a$ and the element $\Gamma_\a:=
\left( \begin{array}{ccc}
1 & 0 \\
\a & 1
\end{array} \right) SL_2(\ZZ)$ of $H$.  The number $\a$ is badly approximable if and only if the forward orbit of $\Gamma_\a$ under the flow $g_t:=
\left( \begin{array}{ccc}
e^{-t} & 0 \\
0 & e^{t}
\end{array} \right)$ misses a neighborhood of $\{\infty\}$.  One can now easily conclude that the set of points in $H$ with bounded forward orbits under $g_t$ has full Hausdorff dimension~\cite{Da}.  A number of generalizations of this theorem exist (see, for example, \cite{Da}, \cite{KM2}, and \cite{Kl}).

For compact manifolds, M. Urba\'nski has a number of results~\cite{Ur}.  One of these is that, for certain Anosov diffeomorphisms, the set of points whose orbits are nondense has full Hausdorff dimension.  Another one of his results, and a chief concern of this note, is 
\begin{theo}\label{thmUrbexpanding}Let $T$ be as above.  If $V$ is a nonempty open subset of $M$, then the Hausdorff dimension of the set of all points contained in $V$ whose forward orbits under $T$ are nondense in $M$ equals $\dim M$.
\end{theo}

The proofs of the various results in~\cite{Ur} are elegant, but they all contain a (essentially the same) minor gap.  There are two corrections of this gap for Theorem~\ref{thmUrbexpanding}.  The first one, by the current author, will be discussed in detail and proved in Subsection~\ref{subsecCorrect} below.  The second one, by Mariusz Urba\'nski, the original author of the theorem, will be outlined in Subsection~\ref{subsecOOP}.  Corrections for the other results should be very similar to these two.

The other chief concern of this note is a result of S. G. Dani concerning certain nondense orbits of endomorphisms of tori.  To state Dani's theorem and one of our results, we must first summarize Schmidt games.

W. Schmidt introduced the games which now bear his name in~\cite{Sch2}.  Let $0 < \a <1$ and $0 < \b<1$.  Let $S$ be a subset of a complete metric space $M$.  Two players, Black and White, alternate choosing nested closed balls $B_1 \supset W_1 \supset B_2 \supset W_2 \cdots$ on $M$.  The radius of $W_n$ must be $\a$ times the radius of $B_n$, and the radius of $B_n$ must be $\b$ times the radius of $W_{n-1}$.  The second player, White, wins if the intersection of these balls lies in $S$.  A set $S$ is called \textbf{$(\a, \b)$-winning} if White can always win for the given $\a$ and $\b$.  A set $S$ is called {\bf $\a$-winning} if White can always win for the given $\a$ and any $\b$.  Schmidt games have four important properties for us~\cite{Sch2}:

\begin{property}[SG1] The sets in $\RR^n$ which are $\a$-winning have full Hausdorff dimension.
\end{property}

\begin{property}[SG2] Countable intersections of $\a$-winning sets are again $\a$-winning.\end{property}

\begin{property}[SG3] If a set is $\a$-winning, then it is also $\a'$-winning for all $0 < \a' \leq \a$.
\end{property}

\begin{property}[SG4]  Let $0 < \a \leq 1/2$.  If a set in a Banach space of positive dimension is $\a$-winning, then the set with a countable number of points removed is also $\a$-winning.
\end{property}

We may now precisely state Dani's result from~\cite{Da2}:

\begin{theo}\label{thmDani}  Let $f$ be a semisimple, surjective linear endomorphism of the torus $\TT^n:= \RR^n /\ZZ^n$ where $n \geq 1$.  The set of points whose forward orbit closures miss the identity element $0$ in $\TT^n$ is $1/2$-winning.
\end{theo}

Finally, we note that there are some interesting results where one considers points whose orbits (eventually) avoid certain uncountable sets~\cite{Do2}.

\subsection{Statement of Results}

We give a correction for the proof of Theorem~\ref{thmUrbexpanding} and note that our proof allows us to show a stronger theorem than the original:  

\begin{theo}\label{thmfullHD1} Let $T$ be as above.  Given $x_1, \cdots, x_p \in M,$ the set of points whose forward orbit closures miss $x_1, \cdots, x_p$ has full Hausdorff dimension (i.e. $=\dim M$).
\end{theo}

\begin{rema}  A similar result, proved using a line of reasoning different from that in~\cite{Ur} or this note, can be found in~\cite{AN}.  The proof in~\cite{AN} uses higher dimensional nets and Kolmogorov's consistency theorem from probability theory, while the proofs in~\cite{Ur} and this note are based on elementary properties of Markov partitions.  See Section~\ref{secConclude} for a discussion of Theorem~\ref{thmfullHD1} and the result in~\cite{AN}.
\end{rema}

Let us now describe the proof scheme of Theorem~\ref{thmUrbexpanding} found in~\cite{Ur}.  The Markov partition associated with $T$ is used to encode the dynamics of $T$ with left shifts of  certain infinite strings.  A point in $M$ corresponds to at least one infinite string, and the action of $T$ on this point corresponds to the action of the left shift operator on this string.  Thus, the Markov partition provides a semi-conjugacy from a subshift of finite type to $(M, T)$.  To avoid open neighborhoods of a point $x_0$, one must construct infinite strings that do not have certain finite ``bad'' strings, which correspond to neighborhoods of $x_0$, as substring.  To finish, a slightly adapted lemma of C. McMullen (which itself is an application of Frostman's lemma) is used to try to show that the set of these infinite strings corresponds to a set of points in $M$ of full Hausdorff dimension~\cite{Ur}.

The minor gap in this proof is that these infinite strings are checked to not have bad strings as substrings only in certain positions.  Other positions are not checked, and hence some of the infinite strings thus constructed will contain bad strings as substring.  The positions that are checked are specific.  The author's correction, then, is to handle checking generic positions, which is made possible by Lemma~\ref{lemmPartMatch} (the No Matching lemma) below.  Urba\'nski's correction, on the other hand, is to replace $T$ by an appropriate power and deduce the result for $T$ using the Baire category theorem.  

While Urba\'nski's correction is much shorter than the author's, the author's correction has an important, felicitous benefit.  By using Lemma~\ref{lemmPartMatch}, one can extract precise information on how to construct these infinite strings:  enough information to play Schmidt games.  Hence, via the author's correction, we obtain our other result:

\begin{theo}\label{thmWin1}  Let $M$ be the circle $S^1:=\RR/\ZZ$ and $T$ be as above.  Given a point $x_0 \in M$, the set of points whose forward orbit closures miss $x_0$ is $\a$-winning for some $0 < \a \leq 1/2$.
\end{theo}

Using the properties of Schmidt games, we obtain 

\begin{coro}\label{coroWin1} Let $\T$ be any finite set of $C^2$-expanding self-maps of $S^1$ and $A \subset S^1$ be any countable set. Then the set of points whose forward orbit closures under any map in $\T$ that miss $A$ is $\a$-winning for some $0 < \a \leq 1/2$.
\end{coro}

\begin{rema}Hence, we have generalized in dimension one Theorems~\ref{thmUrbexpanding} and~\ref{thmfullHD1} (and also the aforementioned result in~\cite{AN}) and (in part) Theorem~\ref{thmDani}.  See Theorem~\ref{thmWin2} and Corollary~\ref{cor1DimWin} below for more precise statements of Theorem~\ref{thmWin1} and Corollary~\ref{coroWin1} respectively.\end{rema}

The gist of the proof of Theorem~\ref{thmWin1} is ``fitted descent.''   On one hand, as one plays a Schmidt game, one forms a nested sequence of closed intervals over which one has partial control.  On the other hand, the desired infinite strings are created recursively from longer and longer finite strings, which form a corresponding nested sequence of closed intervals; partial control for these intervals is given by Lemma~\ref{lemmPartMatch}.  Fitted descent is the idea that one must fit, carefully, these nested intervals together.  In the end, one obtains that the desired set is winning.

\section{The Original Proof}

In this section, we outline Urba\'nski's original proof from~\cite{Ur}, explicitly describe and illustrate with instructive examples the minor gap in this proof, and outline Urba\'nski's correction.  To begin, we must recall the basic properties of Markov partitions and describe a lower bound for Hausdorff dimension.

\subsection{The Basics of Markov Partitions}\label{subsecMPBasics}  Much of this subsection will follow the development in~\cite{Ur}.  Let $M$ and $\sigma$ be as above.  If $A \subset M$, let us denote its topological closure in $M$ by $\overline{A}$.  A $C^1$-map $f:M\rightarrow M$ is \textbf{expanding}\footnote{The definition of expanding only requires that $f$ be $C^1$.  This note, however, considers only $C^2$-expanding maps.} if (perhaps after a smooth change of Riemannian metric) there exists a real number $\lambda >1$ such that \[\|D_xf(v)\| \geq \lambda \|v\|\] for all $x \in M$ and for all $v$ in the tangent space of $M$ at $x$~\cite{KS}.  Recall that our map $T$ is $C^2$-expanding.  A \textbf{Markov partition for $T$} is a finite collection $\Rc := \{R_1, \cdots, R_s\}$ of nonempty subsets of $M$ such that 
\begin{align}
	\label{MP1} & M = R_1 \cup \cdots \cup R_s \\
	\label{MP2} & R_j = \overline{\textrm{Int}R_j} \textrm{ for every } j = 1, \cdots, s \\
	\label{MP3} & \textrm{Int}R_i \cap \textrm{Int}R_j = \emptyset \textrm{ for all } 1 \leq i \neq j \leq s \\
	\label{MP4} & \sigma(R_j \backslash \textrm{Int}R_j) = 0 \textrm{ for all }j= 1, \cdots, s \\
	\label{MP5} & \textrm{For every } j \in \{1, \cdots, s\}, T(R_j) \textrm{ is a union of elements of } \Rc.
\end{align}

The \textbf{diameter} of a Markov partition is the maximum diameter over all its elements.  Because $T$ is expanding, $T$ is injective on any set $A \subset M$ if diam$(A)$ is smaller than a constant $\delta_T>0$.  A Markov partition \textbf{with small diameter} is a Markov partition whose diameter $< \delta_T$.  All Markov partitions in this note have small diameters (except for one, which we will indicate).  We assume, therefore, that $\Rc$ has small diameter.  For a proof of the existence of Markov partitions with small diameters, see~\cite{KS}.   Even more, there exist Markov partitions for $T$ which have diameters as small as one likes (\cite{KS} and~\cite{Ru}).  

Let us enumerate properties of the Markov partition $\Rc$ with small diameter.  The \textbf{transition matrix} of $\Rc$ is the $s \times s$ matrix given by \[A_{i,j}:= \left\{ \begin{array}{ll}
1 & \textrm{if } T(R_i) \cap \textrm{Int}R_j \neq \emptyset\\
0 & \textrm{if } T(R_i) \cap \textrm{Int}R_j = \emptyset
\end{array} \right.\] for any pair $i, j \in \{1, \cdots, s\}$.

We will explicitly describe below how the transition matrix for $\Rc$ allows us to encode the dynamics of $T$ by using sequences of elements of $\{1, \cdots, s\}$.  We will on occasion refer to this encoding and these sequences as the symbolic description (of $T$).  First we need some notation.  We refer to the set $\{1, \cdots, s\}$ as an \textbf{alphabet}\footnote{In this note, we assume that all alphabets have at least two letters.} (or, in particular, the alphabet for $\Rc$) and its elements as \textbf{letters}.  A \textbf{string} is a bi-infinite, infinite, or finite sequence of the letters of the given alphabet.  Thus, every element of a string has at most one predecessor and at most one successor.  A \textbf{valid} string is a string given by the transition matrix $A$ as follows:  for every element $i$ of the string with a successor $j$, $A_{i,j}=1$.  Let $h \leq t$ be integers.  A \textbf{$(h,t)$-string} $\a$ is a string $\a_h \a_{h+1} \cdots \a_t$ with the given indices, and a \textbf{substring} of $\a$ is a string $\a_i \cdots \a_j$ where $h \leq i \leq j \leq t$.  Also, given $h \leq i \leq j \leq t$, the \textbf{$(i,j)$-substring} of $\a$ is the string $\a_i \cdots \a_j$.  An $(i,j)$-string $\gamma$ is a \textbf{(extrinsic) substring} of $\a$ if there exists a $(k,k+j-i)$-substring of $\a$ such that $\a_k=\gamma_i, \a_{k+1} = \gamma_{i+1}, \cdots, \a_{k+j-i} = \gamma_j$.  For convenience, $(0,t)$-strings will also be called \textbf{$t$-strings}.  

For $n \in \NN \cup \{0\}$, let $\Sigma(n)$ denote the set of valid $n$-strings.  For $\a \in \Sigma(n)$, define \[R_\a := R_{\a_0} \cap T^{-1}(R_{\a_1}) \cap \cdots \cap T^{-n}(R_{\a_n}).\]  Thus, for all $n \in \NN \cup \{0\}$, $R_\a \neq \emptyset$ and has the following properties (see~\cite{Ur} and~\cite{KS}):  
\begin{align}
\label{MP6} & \cup_{\a \in \Sigma(n)}R_\a = M \\
\label{MP7} & R_\a = \overline{\textrm{Int}R_\a} \\
\label{MP8} & \textrm{Int}R_\a \cap \textrm{Int}R_\b = \emptyset \textrm{ for every distinct pair } \a, \b \in \Sigma(n) \\
\label{MP9} & T(R_\a) = R_{\a_1 \cdots \a_n} \\
\label{MP10} & T^{-1}(R_\a) = \cup_{\{i \mid A_{i,\a_0}=1\}} R_{i \a} \\
\label{MP11} & R_\a = \cup_{\{i \mid A_{\a_n,i}=1\}} R_{\a i} \\
\label{MP12} & \sigma(R_\a \backslash \textrm{Int}R_\a) = 0 \\
\label{MP13} & \textrm{diam}(R_\a) < \delta_T \lambda^{-n}.
\end{align}

The final property, which follows, is the important bounded distortion property.  Let $J(T)(x) = |\det D_x T| $ denote the Jacobian of $T$ at the point $x$.  Given a Borel set $A \subset M$ on which $T$ is injective, we have \[\sigma(T(A)) = \int_A J(T) d\sigma.\]  The bounded distortion property is the following theorem (\cite{Ur} and see~\cite{KS} for a proof):
\begin{theo} \label{thmBDI}
There exists a constant $C \geq 1$ such that \[\frac{J(T^n)(y)}{J(T^n)(x)} \leq C\] for all $n\geq1$, $\a \in \Sigma(n)$,  and $x, y \in R_\a$.
\end{theo}

\noindent This theorem will allow us to bound ratios of volumes as Lemma~\ref{lemmBDII} below shows.

Finally, let $\Sigma(\infty)$ denote the set of valid infinite strings $\a_0 \a_1 \cdots$ indexed by $\NN \cup \{0\}$.  If $\a \in \Sigma(\infty)$, then $R_\a$ is a unique point in $M$.  Conversely, if $x \in M$, then there exists an $\a \in \Sigma(\infty)$ such that $x = R_\a$.  Therefore, we have a semi-conjugacy from the subshift of finite type $(\Sigma(\infty), \textrm{ the left shift operator})$ onto $(M, T)$.  A \textbf{representation} of $x \in M$ is an element $\a \in \Sigma(\infty)$ for which $x = R_\a$.  A representation may not be unique.

\subsection{A Lower Bound for Hausdorff Dimension} In this subsection, we follow a simplified version of the development in~\cite{Ur}.  Let $K \subset M$ be compact.  For $k \in \NN$, let $E_k$ denote a finite collection of compact subsets of $K$ with positive volume. (Recall that volume measure is denoted by $\sigma$.)  We require the following to hold:
\begin{align}
\label{MM1} & \textrm{The union of the elements of } E_1 \textrm{ is } K. \\
\label{MM2} & \textrm{For distinct } F, G \in E_k, \sigma(F \cap G) = 0. \\
\label{MM3} & \textrm{Every element }F \in E_{k+1} \textrm{ is contained in an element } G \in E_k.	
\end{align}
Let us define the following notation:
\begin{itemize} 
\item Let $\cup E_k$ denote the union of all elements of $E_k$.
\item Let $E := \cap_{k=1}^\infty \cup E_k.$
\item Define, for every $F \in E_k$, \[\textrm{density}(E_{k+1}, F):= \frac{\sigma(\cup E_{k+1} \cap F)}{\sigma(F)}.\]
\item Let $\Delta_k := \inf\{\textrm{density}(E_{k+1}, F) \mid F\in E_k\}.$
\item Let $d_k := \sup\{\textrm{diam}(F) \mid F \in E_k\}$.
\end{itemize}  

We further require the following to hold:
\begin{align}
\label{MM4} & \Delta_k > 0 \\
\label{MM5} & d_k < 1 \\
\label{MM6} & \lim_{k \rightarrow \infty} d_k =0.
\end{align}

Following~\cite{Kl}, let us call $\{E_k\}_{k \in \NN}$ a \textbf{strongly tree-like} collection.  Let $HD(\cdot)$ denote Hausdorff dimension.  The following lemma for this strongly tree-like collection is proved in~\cite{Ur} by adapting a proof from~\cite{MM} (both proofs are based on Frostman's lemma):

\begin{lemm} \label{lemmMMUr} It holds that \[HD(E) \geq \dim M - \limsup_{k \rightarrow \infty} \frac{\sum_{j=1}^k \log \Delta_j}{\log d_k}.\]
\end{lemm}

\begin{rema}  The upper index of summation is $k$ (not $k-1$ as in~\cite{Ur}).  See~\cite{Kl}, but note that what is referred to as the ``$j$-th stage density'' must be $> 0$.  For a more general version of this lemma, see~\cite{Kl} or~\cite{Ur}.  For a version involving higher dimensional nets, see~\cite{AN}.
\end{rema}

\subsection{An Outline of the Original Proof}\label{subsecOOP}  In this subsection, we setup the proof of Theorem~\ref{thmUrbexpanding} following pages 390-391 of~\cite{Ur}, explicitly describe and illustrate with three types of examples the minor gap, and outline Urba\'nski's correction.

Fix a Markov partition $\Rc$ whose diameter is small enough.  Fix a $q \in \NN$ large enough, and fix a $\gamma \in \Sigma(q)$ such that $\gamma_0 \neq 1$.  Define \begin{align} E_k := E_k(q) := \{& R_\a \mid \a \in \Sigma(kq), \a_0 = 1, \textrm{ and } T^n(R_\a) \cap \textrm{Int}R_\gamma = \emptyset \nonumber \\ & \textrm{for every } n=0, 1, \cdots, (k-1)q\}.\nonumber\end{align}  Hence, $R_\a \in E_k$ if and only if $\gamma$ is not a substring of $\a
$ and $\a_0=1$.\footnote{That $\a_0=1$ is an unimportant detail.  It merely allows us to consider only points in the nonempty open set $V$.}  

The outline of the proof in~\cite{Ur} and the proof of the author's correction below should now be evident.  Let $K := \cup E_1$.  We should show that $\{E_k\}$ is strongly tree-like so that we can apply Lemma~\ref{lemmMMUr} to obtain a lower estimate for the Hausdorff dimension of $E(q) := \cap_{k=1}^\infty \cup E_k.$  Verifying (\ref{MM2}), (\ref{MM3}), (\ref{MM5}), and (\ref{MM6}) is routine and can be found in~\cite{Ur}.

The minor gap in the proof in~\cite{Ur} arises from incorrectly estimating $\Delta_k$.  It is asserted that (\cite{Ur}, p 390), for $R_{\a i} \in E_k,$ \begin{subequations}\label{badassert} \begin{align} & R_{\a i} \backslash \cup E_{k+1} = \emptyset \textrm{ whenever } i \neq \gamma_0 \textrm{ and } \label{badasserta}\\ & R_{\a i} \backslash \cup E_{k+1} = R_{\a \gamma} \textrm{ whenever } i = \gamma_0.\label{badassertb}\end{align}\end{subequations}  This assertion, however, is false as the following three examples will show.  Because this assertion is false, the estimate of $\Delta_k$ is not always large enough to prove the theorem (see the third example below for an explicit demonstration).  Thus, a correction must be made or else the original proof will not prove the theorem.  Moreover, in addition to demonstrating the minor gap in the original proof, these three examples will illustrate how to formulate the author's correction below.  What is needed, as we shall see in Section~\ref{secCor}, is to consider other elements in $\a i$ in addition to the last element $i$.

The examples can be constructed on a very simple system $(M, \mu, T)$ where $M = S^1:=\RR / \ZZ$, $\mu = \sigma$ is the probability Haar measure, and $$T = m_2:S^1 \rightarrow S^1; x \mapsto 2x.$$

Let $R_1 = [0, 1/2]$ and $R_2 = [1/2, 1]$. It is clear that $\{R_1, R_2\}$ is a Markov partition\footnote{Technically, this Markov partition does not have small diameter.  But, the lack of injectivity on an element of this partition occurs only at the two endpoints, which we can ignore with impunity.  Moreover, the reader can construct similar examples for any dyadic partition with diameter as small as desired.  See the third example below for such a construction.} for $T$ and that the associated transition matrix is $$
\left( \begin{array}{ccc}
1 & 1 \\
1 & 1
\end{array} \right)
.$$  

\begin{proof}[The First Example]
Let $q=2$, $k=2$, $\gamma = 211$, $\a = 1222$, and $i=1$.  Now $R_{\a i} \in E_2$, but $R_{\a i 1  *} \notin E_3$ (where $*$ is any element of $\{1, 2\}$).  Since $R_{\a i} \supset R_{\a i 1 *}$, $R_{\a i} \backslash \cup E_3 \neq \emptyset$, which contradicts (\ref{badasserta}) and finishes the example. \end{proof}

Almost identical examples can be fashioned for any $k$ and $q$ large enough. Similar examples can be fashioned for other Markov partitions and other expanding systems.  It is, now, easy to infer that the gist of the problem is that while $R_\b \in E_N$ implies that $\gamma$ is not a substring of $\b$, one must still consider the case where a truncated $\gamma$ string is the tail of $\b$.  This statement will be made precise in Subsection~\ref{subsecMoreStrings}.

Handling these strings can be complex:

\begin{proof}[The Second Example]
Let $q=4$, $k=2$, $\gamma = 21211$, and $\a = 111112121$.  Hence, $R_\a \in E_2$, but both $R_{\a1***} \notin E_3$ and $R_{\a211*} \notin E_3$ (where each instance of $*$ is some, possibly distinct, element of $\{1, 2\}$). This contradicts (\ref{badasserta}).\end{proof}

Finally, the third example, below, not only contradicts (\ref{badassert}), but also when used in the remainder of the proof in~\cite{Ur} (p 390-391) gives the following true, but not useful fact:  $HD(E(q)) \geq 0$.  First, consider the dyadic partition \[\D=\{R_1, \cdots, R_{2^s}\}\] where $$R_i = [\frac {i-1} {2^s}, \frac i {2^s}].$$  

It is clear that $\D$ is a Markov partition for $T = m_2$ with transition matrix $$ A:=\left( \begin{array}{c}
B \\
B
\end{array} \right)$$ where $$ B := 
\left( \begin{array}{cccccccccccc}
1 & 1 & 0 & 0 & \cdots & & & & & 0 & 0\\
0 & 0 & 1 & 1 & 0 & 0 & \cdots & & & 0 & 0 \\
0 & 0 & 0 & 0 & 1 & 1 & 0 & 0 & \cdots & 0 & 0 \\
&&&&& \vdots &&&&& \\
0 & 0 & \cdots &&&&& 0 & 0 & 1 & 1
\end{array} \right)
.$$ 

Furthermore, by the choice of $s$, we can make the diameter of $\D$ be as small as we like.  For convenience, let $a:=2^s$, $b:=2^s-1$, and $c:=2^{s-1}$.  Note that $A_{c,a}=1$.  Also, for $n \in \NN \cup \{0\}$ and every $\a \in \Sigma(n)$, $R_\a$ is an interval, and $\sigma(R_\a) = 2^{-(s+n)}$.

\begin{proof}[The Third Example]
Choose $s$ large enough so that the diameter of $\D$ is small enough.  Let $q$ be large and $\gamma = a \cdots ab$.  Let $k$ be any natural number.  Without loss of generality, let us exchange $R_c$ with $R_1$.  Let $\a = 1a \cdots a$; then $R_\a \in E_k$.  Let $\b \in \Sigma(q-1)$; then, by (\ref{MP11}), the only $R_{\a\b} \in E_{k+1}$ is for $\b = a \cdots a$.  This contradicts (\ref{badassertb}). 

Let us continue to follow the proof in~\cite{Ur} (p 390-391) using the setup of this example.   We note that \[\sigma(\cup E_{k+1} \cap R_\a) = \sigma(R_{\a a \cdots a}) = \frac{\sigma(R_\a)}{2^q}.\] Thus, \[\textrm{density}(E_{k+1}, R_\a) = 2^{-q}.\]  Let $R_\eta$ be any element of $E_k$.  If $\eta_{kq} = a$, then $R_{\eta \b} \in E_{k+1}.$  Otherwise, if $\eta_{kq} \neq a$, then any valid concatenation (of the correct length) of $\eta$ will produce an element of $E_{k+1}$.  Thus, \[\textrm{density}(E_{k+1}, R_\eta) \geq 2^{-q},\] and $\Delta_j = 2^{-q}$.

Hence, $$\limsup_{k \rightarrow \infty} \frac {\sum_{j=1}^k \log \Delta_j} {\log d_k} \leq \limsup_{k \rightarrow \infty} \frac{-qk \log 2}{\log \delta_T -qk \log 2}= 1.$$  

Applying Lemma~\ref{lemmMMUr} as in~\cite{Ur}, we obtain  $$HD(E(q)) \geq 0,$$ a true, but not useful fact.  It is no help in showing the desired full Hausdorff dimension assertion.   
\end{proof}

Hence, in (\ref{badasserta}), it is possible for \[R_{\a i} \backslash \cup E_{k+1} \supsetneq \emptyset\] when $i \neq \gamma_0$, and, in (\ref{badassertb}), it is possible for \[ R_{\a i} \backslash \cup E_{k+1} \supsetneq R_{\a \gamma}\] when $i = \gamma_0.$  Thus, other elements of $\a i$ besides the last element $i$ must be considered in order to correctly estimate $\Delta_k$.  This is the author's approach and will be discussed in detail in Section~\ref{secCor}.  The other, Urba\'nski's approach, is to replace $T$ with $T^q$ and thus to avoid the need to consider other elements of $\a i$.  Let us outline Urba\'nski's approach now.

After the author alerted Urba\'nski to the minor gap by summarizing the examples and the gist of Section~\ref{secCor}, Urba\'nski provided the following alternative correction in personal communication.  Instead of considering the set $E_k(q)$ above, consider the set  \[\tilde{E}_k(q) := \{R_\a \mid \a \in \Sigma(kq), \a_0=1, \textrm{ and } T^{qn}(R_\a) \cap \textrm{Int}R_\gamma = \emptyset \]\[\textrm{ for every } n = 0, 1, \cdots k-1\}\] where $q$ is a large enough natural number and $\gamma \in \Sigma(q)$ with $\gamma_0 \neq 1.$  Clearly, $\tilde{E}(q):=\cap_{k=1}^\infty \cup \tilde{E}_k(q)$ is the set of points whose forward orbits under $T^q$ (instead of $T$) avoid Int$R_\gamma$.  Verifying (\ref{MM1}), (\ref{MM2}), (\ref{MM3}), (\ref{MM5}), and (\ref{MM6}) is still routine, and, what is more, it is clear (since we iterate by $T^q$) that now (\ref{badassert}) also holds. To show that $\Delta_k \geq 1/2$ (for $\tilde{E}_k(q))$, one can now follow the proof exactly as in~\cite{Ur}, namely one uses the bounded distortion principle as expressed in the form of Lemma~\ref{lemmBDII} below.  Applying Lemma~\ref{lemmMMUr} yields \[HD(\tilde{E}(q)) \geq \dim M + \frac {- \log 2}{q\log\lambda}.\]  

Now let $x$ be a point in $\tilde{E}(q)$.  Since $\{T^{qn}(x)\}_{n=1}^\infty$ is nondense, it is nowhere dense by ergodicity.  The same remark applies to its iterates $\{T^{qn-j}(x)\}_{n=1}^\infty$ for every $j = 1, \cdots, q$.  Thus, if the orbit of $x$ under $T$ itself were dense, we would obtain a contradiction of the Baire category theorem.  Hence, $\tilde{E}(q)$ is also a set of points whose forward orbits under $T$ are nondense.  Letting $q \rightarrow \infty$ shows that the set of points whose forward orbits under $T$ are nondense has full Hausdorff dimension. 

Clearly, Urba\'nski's correction is a concise, elegant perturbation of his original proof.  This conciseness, however, loses information.  In particular, we no longer, after applying the Baire category theorem, keep track of the point that these forward orbit closures miss, nor do we have much information on the symbolic description (as infinite strings) of these points whose forward orbits are nondense.  Both pieces of information are useful, but the second piece, the symbolic description of the points with nondense orbits, is crucial to the generalization of this result in dimension one to Schmidt games.  

On the other hand, the author's correction, while it is longer and more involved, is able to keep track of the point being missed and provides some precise information on the symbolic description of the points with nondense orbits, enough information to play Schmidt games.  In the next section, we turn to the author's correction, which is based on the insight gleaned from the examples above.  These examples  illustrate all the different types of difficulties that we will encounter in making the correction. 

\section{The New Proof} \label{secCor}

In this section, we provide the author's correction for the proof of Theorem~\ref{thmUrbexpanding}.  Before giving the proof of the correction, we must make a more in-depth study of strings and use the insight gleaned from this study to prove the ``No Matching'' lemma (Lemma~\ref{lemmPartMatch} below) used in the correction.  The last step before proving the correction is to study Markov partitions further and, in particular, make our first refinement of the bounded distortion property.

\subsection{More on Strings}\label{subsecMoreStrings} Recall the basic facts about strings from Subsection~\ref{subsecMPBasics}.  There are some more basic facts that we need.  Given a string with an element $i$ that has no successor, a \textbf{concatenation or appending (on the right)} is a new string identical to the given string except that a successor is chosen from the given alphabet for this element $i$.  (Note that a repeated concatenation may be referred to simply as a concatenation depending on context.)  Given any string $\a$, define the \textbf{length of $\a$, $l(\a)$}, to be the number of elements in $\a$.  A string is \textbf{finite} if it is a finite sequence.  A finite string is \textbf{reducible} if it is of the form $a \cdots a$ where $a = \a_0 \cdots \a_r$ is a string of length $r+1$.  A finite string is \textbf{irreducible} if it is not reducible.   

\subsubsection{Partial String Matches}  Let $n \leq N$.  A $n$-string $\b$ is \textbf{equivalent} to a $N$-string $\a$ (or a $N$-string $\a$ is \textbf{equivalent} to a $n$-string $\b$) if $\a_0=\b_0, \cdots, \a_n = \b_n$.  Let $\gamma$ be a $n$-string and $\a$, a $N$-string.  A \textbf{match of $\gamma$ with $\a$} is an $(i, i+n)$-substring of $\a$ given by $\a_i = \gamma_0, \a_{i+1} = \gamma_1, \cdots, \a_{i+n} = \gamma_n$.  Whenever $\gamma$ is a substring of $\a$, there is at least one such match.  A  \textbf{partial match of $\gamma$ with $\a$} is an $(i,N)$-substring of $\a$ given by $\a_i = \gamma_0, \a_{i+1} = \gamma_1, \cdots, \a_N = \gamma_m$ where $m < n$.  Consequently, $i > N-n$.  Call $i$ the \textbf{head (of the partial match)}.

Note that if two partial matches of $\gamma$ with $\a$ have heads $i <j $, then a ``right shift and crop'' of the one with the smaller head will produce the one with the larger head.  This is just pattern matching.

\subsubsection{Valid Strings and Matching.}  Let us now specialize to valid strings (defined in Subsection~\ref{subsecMPBasics}) for a Markov partition with small diameter $\Rc:=\{R_1, \cdots, R_s\}$.  

By (\ref{MP5}) and (\ref{MP11}), there exists a letter for which concatenation on the right of any valid finite string produces a valid finite string.  But, there exist Markov partitions such that for some letter $i$, only one letter $j$ produces a valid $1$-string when concatenated on the right; such $i$ is called a \textbf{degenerate letter}.  A letter that is not degenerate is \textbf{nondegenerate}.  A \textbf{block} of a string is a substring composed of exactly one nondegenerate letter, which is found at the largest index.  Note that given the initial letter in a block, the only valid concatenation on the right of the initial letter is the one that produces the rest of the block.  By (\ref{MP13}), there exists an integer $B$, called the \textbf{maximal block length}, such that for every $B$-string $\a$, $\sigma(R_{\a_0}) > \sigma(R_\a)$.  A \textbf{general block} of a string is a substring composed of exactly one nondegenerate letter. A \textbf{reverse block} of a string is a substring composed of exactly one nondegenerate letter, which is found at the smallest index.  A \textbf{double general block} of a string is a substring composed of a block followed by a reverse block.  (Hence, a double general block has exactly two nondegenerate letters; they are adjacent.)

\begin{lemm}\label{lemmDistinctDisjoint} In a string composed of only degenerate letters, each letter is distinct.
\end{lemm}

\begin{proof} Assume not.  Let $\a$ be a string of degenerate letters, and let $\a_i = \a_j $ for $i < j$.  The only valid concatenations of $\a$ are those for which the following substring repeats:  $\a_i \a_{i+1} \cdots \a_{j-1}.$  Let $\gamma$ be a concatenation of $\a$ longer than the maximal block length.  But $\sigma(R_\gamma) = \sigma(R_{\a_0})$, a contradiction. \end{proof}

\begin{coro}  \label{corMaxBlock}The maximal block length of the Markov partition $\{R_1, \cdots, R_s\}$ is at most $s$.  The maximal length of any general block is at most $2s -1$; any double general block, at most $2s$.
\end{coro}

\begin{proof} The lemma implies that the maximal block length is at most $s+1$.  If it equals $s+1$, then all letters are degenerate, a contradiction of (\ref{MP13}). \end{proof}

We are now ready to prove the key lemma in the proof of the correction and of the generalization:

\begin{lemma}[No Matching]\label{lemmPartMatch}Let $N \geq n \geq 8s-4$.  Let $\gamma$ be any $n$-string such that $\gamma_{n-1}$ is nondegenerate \textbf{except} those of the following kind:  \[\gamma = a^0 \cdots a^m \] where \[a^0 = \cdots = a^{m-1}\] are general blocks and either \begin{equation}\label{eqn1}a^m \textrm{ is a general block not equivalent to }
a^0a^0\end{equation} or \begin{equation}\label{eqn2} a^m \textrm{ is a double general block not equivalent to }
a^0a^0.\end{equation}  And let $\a$ be a $N$-string such that no match of $\gamma$ with $\a$ exists.  Then there exists a choice of substrings $b^0$ and $b^1$ of length at most $s$ such that for any letters $\b_0, \b_1, \cdots \b_k$, no match of $\gamma$ with the $N+n$-string $\a b^0 b^1 \b_0 \cdots \b_k$ exists.
\end{lemma}

\begin{rema}  It is possible for both (\ref{eqn1}) and (\ref{eqn2}) to hold for the same string $\gamma$.
\end{rema}

\begin{rema}
If no match of $\gamma$ with $\a b^0 b^1 \b_0 \cdots \b_k$ exists, then no match of $\gamma$ with $\a b^0 b^1 \b_0 \cdots \b_{k'}$ for any $0 \leq k' \leq k$ exists.
\end{rema}

\begin{proof}  By Corollary~\ref{corMaxBlock}, all $n$-strings contain at least four general blocks.  Note that for the exceptional $n$-strings, we obtain $m \geq 3$ by Corollary~\ref{corMaxBlock}. 

There are three cases: \bigskip

\noindent\textbf{Case 1:  No partial matches of $\gamma$ with $\a$ exist.} \medskip

Choose any letters for $b^0 b^1 \b_0 \cdots \b_k$ that make $\a b^0 b^1 \b_0 \cdots \b_k$ a valid string.  \bigskip

\noindent\textbf{Case 2:  There exists exactly one partial match of $\gamma$ with $\a$.} \medskip

Hence, there exists exactly one choice of letter for the initial letter of $b^0$ (namely a choice for $b^0_0$) which would produce a match of $\gamma$ with $\a b^0 b^1 \b_0 \cdots \b_k$.  

If $\a_N$ is nondegenerate, let $b^0_0$ not be this letter.  Since no other partial matches of $\gamma$ with $\a$ exists, we are free to take any letters for the remainder as long as they form a valid string.  

If $\a_N$ is degenerate, then $\a_Nb^0$ must contain the block starting with $\a_N$.  Also, because $\gamma_{n-1}$ is nondegenerate, $\gamma$ must contain this block (not just partially contain it).  By the definition of block, we have a choice of letter to concatenate on the right of the block.  Choose the letter that is different from what is in $\gamma$.  Hence, we have made a choice for $b^0$, and again we are free to take any letters for the remainder as long as they form a valid string.  

\bigskip

\noindent\textbf{Case 3:  There are at least two partial matches of $\gamma$ with $\a$.}  \medskip

Let $i$ be the smallest head and $j$ be the second smallest head.  Let $\gamma^i$ correspond to the partial match with $i$ as head; $\gamma^j$, with $j$ as head.

Now $\gamma^i$ is the concatenation of the same substring of length $j-i\geq 1$. Denote the substring by $c = \gamma_0 \cdots \gamma_{j-i-1}$.  Then,  $$\gamma^i = c \cdots c \gamma_0 \cdots \gamma_r$$ where $0 \leq r \leq j-i-1$.  (Note that $\gamma^i$ may contain only one substring $c$.)  For an nonnegative integer $t$, let $0 \leq \bar{t} < j-i$ denote the representative of $t$ mod $j-i$.  By Lemma~\ref{lemmDistinctDisjoint}, $c$ contains at least one nondegenerate letter.  (Moreover, $i$ and $j$ imply that $c$ is irreducible.)  There are two cases:  \bigskip

\noindent\textbf{Case 3A:  The substring $c$ is a general block.} \medskip

If $\gamma_r$ is the one nondegenerate letter in $c$, then choose $b^0$ to be a $0$-string and $b^0_0 \neq \gamma_{\overline{r+1}}$.  Now, if $b^0_0$ is nondegenerate, we have that $\gamma^i b^0 = c \cdots c \tilde{c}$ where $\tilde{c}$ is a double general block.  Otherwise, if $b^0_0$ is degenerate, we have that $\gamma^i b^0 = c \cdots c \tilde{c}$ where $\tilde{c}$ is a general block.    

Otherwise, $\gamma_r$ is degenerate, and thus the choices of $b^0$ are fixed until after we reach the next nondegenerate letter, $b^0_q$.  Because $\gamma_{n-1}$ is nondegenerate, $\gamma^i b^0_0 \cdots b^0_q$ is a substring of $\gamma$.  Moreover, $\gamma_r b^0_0 \cdots b^0_q$ must appear together in $\gamma^i$ because of the repeating substring.  Hence, it is a substring of $cc$.  If $q+1 = l(c)$, then $\gamma_r = b^0_q$, a contradiction.  Hence, $q+1 < l(c)$, and we can choose $b^0_{q+1}$ to be different from the letter that follows the substring $\gamma_r b^0_0 \cdots b^0_q$ in $cc$.  Now, if $b^0_{q+1}$ is nondegenerate, we have that $\gamma^i b^0 = c \cdots c \tilde{c}$ where $\tilde{c}$ is a double general block.  Otherwise, if $b^0_{q+1}$ is degenerate, we have that $\gamma^i b^0 = c \cdots c \tilde{c}$ where $\tilde{c}$ is a general block.

Hence, only $\gamma^i$ may possibly be completed to a match of $\gamma$ with $\a b^0 b^1 \b_0 \cdots \b_k$.  If $\gamma^i b^0$ is not equivalent to $\gamma$, we are done.  

Thus, let $\gamma^i b^0$ be equivalent to $\gamma$.  There are two cases.  If $\gamma^i b^0 = c \tilde{c}$, then there is another nondegenerate letter after $\gamma^i b^0$ in $\gamma$ because $\gamma$ contains at least four general blocks.
Otherwise, $\gamma^i b^0 = c \cdots c \tilde{c}$ where $\tilde{c}$ is either a general or double general block not equivalent to $cc$ (as constructed above).  Thus, there is at least another nondegenerate letter after the substring $\tilde{c}$ in $\gamma$ because we exclude strings of the form (\ref{eqn1}) and (\ref{eqn2}).  

Let $\tilde{c}_t$ be the last letter of $\tilde{c}$ (i.e. $\tilde{c}_t = b^0_{q+1}$).  If $\tilde{c}_t$ is nondegenerate, choose $b^1$ to be a $0$-string where $b^1_0$ is a different letter than what follows $\gamma^i b^0$ in $\gamma$.  If $\tilde{c}_t$ is degenerate, then the substring $\tilde{c}_t b^1_0 \cdots b^1_p$ up to the next nondegenerate letter (i.e. $b^1_p$) in $\gamma$ is  determined.  (Since $\gamma_{n-1}$ is nondegenerate, $b^1_p$ comes before $\gamma_n$.)  Since this is a block, we have a choice of letters for $b^1_{p+1}$; pick it so that it is different from that in $\gamma$.  Hence, $\gamma^i$ cannot produce a match either.  We may now pick any letters for the remainder as long as they produce a valid string.  

\bigskip

\noindent\textbf{Case 3B:  The substring $c$ is not a general block.} \medskip

Hence, $c$ contains at least two nondegenerate letters (not necessarily distinct).  Also, $j-i \geq 2$.

If $\gamma_r$ is a nondegenerate letter in $c$, then choose $b^0$ to be a $0$-string and $b^0_0 \neq \gamma_{l(\gamma^i)}$.  Otherwise, the choices of $b^0$ are fixed until after we reach the next nondegenerate letter, $b^0_q$.  Because $\gamma_{n-1}$ is nondegenerate, $\gamma^i b^0_0 \cdots b^0_q$ is a substring of $\gamma$.  Pick $b^0_{q+1}$ to be different from the letter that follows $\gamma^i b^0_0 \cdots b^0_q$ in $\gamma$.   Hence, $\gamma^i b^0$ cannot complete to a match.  

Because of the repeating substrings, we know the $j-i$ adjacent letters, namely a substring of $cc$, that are needed for $\gamma^j$ or any partial match with larger head to produce a match of $\gamma$ with $\a b^0 b^1 \b_0 \cdots \b_k$.  (Note that every letter of $c$ appears in such a substring.)

Let $k$ be any head greater than or equal to $j$.  Assume $\gamma^k b^0$ can be completed to $\gamma$.

Now if $\gamma_r$ is nondegenerate, then $b^0 = \gamma_{\overline{r+1}}$, a $0$-string.  And one of the letters $\gamma_{\overline{r+1}}, \cdots, \gamma_{\overline{r+j-i-1}}$ is also nondegenerate.  If $b^0$ is nondegenerate, choose $b^1_0$ to be different from $\gamma_{\overline{r+2}}$.  Otherwise, $b^0$ is degenerate, and thus the block beginning with $b^0$ must be at most $l(c)-1$ in length.  There are at least two ways to concatenate a letter to the end of this block.  Pick, for $b^1$, one that is different from the one in $\gamma$.

Otherwise, $\gamma_r$ is degenerate.  Assume that $b^0_{q+1}$ is nondegenerate.  Hence, the $q+1$-string $b^0$ has length strictly less than $l(c)$ since otherwise $b^0_{q+1} = \gamma_r$, a contradiction.  Thus, for $\gamma^k b^0$ to complete to $\gamma$, we know, by the length, exactly the letter that is required to be concatenated on the right.  Let $b^1$ be a $0$-string such that $b^1_0$ is not this letter.  Otherwise, $b^0_{q+1}$ is degenerate, and $b^0$ has exactly one nondegenerate letter, namely $b^0_q$.  Choose the beginning of $b^1$ to be the rest of the block starting with $b^0_{q+1}$.  Let $b^1_h$ be the nondegenerate letter at the end of this block.  If $q+2 + h+ 1 = l(c)$, then $b^0b^1_0 \cdots b^1_h$ are all the letters in $c$, and therefore $b^1_h = \gamma_r$, a contradiction.  Thus, for $\gamma^k b^0b^1_0 \cdots b^1_h$ to complete to $\gamma$, we know, by the length, exactly the letter that is required to be concatenated on the right.  Let $b^1_{h+1}$ not be this letter.  We may now pick any letters for the remainder as long as they produce a valid string.  

The proof of the lemma is complete.  \end{proof}

\subsection{More Markov Partitions.}\label{subsecMoreMP}  In this subsection, we continue from Subsection~\ref{subsecMPBasics} our study of Markov partitions.  In particular, we make our first refinement of the bounded distortion property.  

As in Subsection~\ref{subsecMPBasics}, let us consider a Markov partition with small diameter $\Rc := \{R_1, \cdots, R_s\}$ for $T$.  Define \[G(n):=\{R_\a \mid \a \in \Sigma(n)\};\] call $G(n)$ the \textbf{$n^{th}$ generation of $\Rc$}.  Hence, $G(0) = \Rc$.  If $\gamma$ is a valid string, let $G_\gamma$ denote the \textbf{generation that $R_\gamma$ belongs to}.  Also, denote the \textbf{set of boundary points of all elements of all generations of $\Rc$} by $\partial(\Rc)$ (or $\partial$, if the context implies the Markov partition).  Finally, \textbf{interior points} are the points in the full volume set $M \backslash \partial$.

Define a \textbf{lower constant of bounded distortion}

$$\varepsilon(q) := \min_{\delta \in \Sigma(q)} \frac{\sigma(R_\delta)} {\sigma(R_{\delta_0})} > 0$$

and an \textbf{upper constant of bounded distortion}

$$1 \geq \E(q) := \max_{\delta \in \Sigma(q)} \frac{\sigma(R_\delta)} {\sigma(R_{\delta_0})} > 0.$$

It is clear that both $\E(q)$ and $\varepsilon(q)$ are weakly monotonically decreasing functions of $q$, both of which tend to $0$ as $q$ tends to $\infty$.

Following, essentially,~\cite{Ur}, we put the notion of bounded distortion, Theorem~\ref{thmBDI}, into a very useful form (note that $C$ is from the theorem):

\begin{lemm}\label{lemmBDII} For every element $R_\a \in G(N)$ and every element $R_{\a \b} \in G(N+n)$, \[\frac {\varepsilon(n)}{C} \leq \frac {\sigma(R_{\a\b})} {\sigma(R_\a)} \leq C \E(n).\]
\end{lemm}

\begin{proof}  The proof for $N=0$ is obvious.  Let $N \geq 1$.  The proof of the first inequality is the following:   \[\varepsilon(n) \leq \frac{\sigma(R_{\a_N\b})} {\sigma(R_{\a_N})} = \frac{\int_{R_{\a\b}} J(T^{N})(x) d\sigma(x)}{\int_{R_{\a}} J(T^{N})(x) d\sigma(x)} \] \[\leq \frac{\sigma(R_{\a\b}) \sup\{J(T^{N})(x) \mid x \in R_\a\}}{\sigma(R_{\a}) \inf\{J(T^{N})(x) \mid x \in R_\a\}} \leq C \frac{\sigma(R_{\a\b})}{\sigma(R_{\a})}.\]

The second inequality is similar:  \[\E(n) \geq \frac{\sigma(R_{\a_N\b})} {\sigma(R_{\a_N})} = \frac{\int_{R_{\a\b}} J(T^{N})(x) d\sigma(x)}{\int_{R_{\a}} J(T^{N})(x) d\sigma(x)} \] \[\geq \frac{\sigma(R_{\a\b}) \inf\{J(T^{N})(x) \mid x \in R_\a\}}{\sigma(R_{\a}) \sup\{J(T^{N})(x) \mid x \in R_\a\}} \geq \frac 1 C \frac{\sigma(R_{\a\b})}{\sigma(R_{\a})}.\qedhere\]\end{proof}

\noindent We will again refine the notion of bounded distortion in Subsection~\ref{subsecEvenMoreMP}.  

Next, we show that a point has more than one representation (defined in Subsection~\ref{subsecMPBasics}) if and only if it lies on the boundary of $R_\a$ for some valid finite string $\a$:

\begin{lemm}\label{lemmReps}  A point $x \in M$ has non-unique representations $\iff x \in \partial$.  The set of points with non-unique representations is $\sigma$-null. 
\end{lemm}

\begin{proof}  Let $\a$ and $\b$ be two distinct representations of $x$.  Then there exists a least $n \in \NN \cup \{0\}$ such that $\a_n \neq \b_n$.  By (\ref{MP8}), $$\textrm{Int}R_{\a_0 \cdots \a_n} \cap \textrm{Int}R_{\b_0 \cdots \b_n} = \emptyset.$$ Since $x \in R_{\a_0 \cdots \a_n} \cap R_{\b_0 \cdots \b_n}$, it must lie on the boundary.

Let $x$ be on the boundary of $R_\a$ for $\a \in \Sigma(n)$ for some $n \in \NN \cup \{0\}$.  Since $M$ is a manifold, every open neighborhood of $x$ intersects $R_\a$ and $R_\a^c$.  By (\ref{MP6}), $\cup_{\b \in \Sigma(n) \backslash \{\a\}} R_\b \supset R_\a^c$.  Now assume that there exists an open neighborhood $U$ of $x$ such that $U \cap R_\b = \emptyset$ for all $\b \in \Sigma(n) \backslash \{\a\}$.  Hence, $U \cap R_\a^c = \emptyset$, a contradiction.  Thus, $x \in R_\b$ for some $\b \in \Sigma(n) \backslash \{\a\}$.  Since $R_\a$ and $R_\b$ have disjoint interiors, $x$ lies on the boundary of $R_\b$.  Then, by (\ref{MP11}), for every $p \in \NN \cup \{0\}$, there exists a valid string $a^p:=\a \gamma_0 \cdots \gamma_p$ such that $x$ is on the boundary of $R_{a^p}$.  Similarly for $\b$, we obtain $R_{b^q}$ for all $q \in \NN \cup \{0\}$.  Since $\a \neq \b$, $a_\infty \neq b_\infty$.

The boundary of all elements of all generations of the Markov partition has zero measure by (\ref{MP12}).  This is a countable union and hence measure zero.
\end{proof}

Finally, we have (recall that we denote the $(0,Q)$-substring of a $\gamma \in \Sigma(\infty)$ by $\gamma(Q)$):

\begin{lemm}\label{lemmBndPoints}  Let $x \in M$ be a point with representations $\gamma^1, \cdots, \gamma^r.$  Then, for every $Q \in \NN \cup \{0\}$, there exists an open neighborhood $U$ of $x$ such that $U \subset \cup_{t=1}^r \textrm{Int}R_{\gamma^t(Q)} \cup \partial$.
\end{lemm}

\begin{proof} Assume not.  Then there exists a $Q$ such that for any open neighborhood $U$ of $x$, $$U \backslash (\cup_{t=1}^r \textrm{Int}R_{\gamma^t(Q)} \cup \partial) \neq \emptyset.$$  Let $\{U_n\}$ be a family of shrinking balls centered at $x$ for which the above holds.  Thus, there exists $y_n \in U_n$ such that $y_n \notin \cup_{t=1}^r \textrm{Int}R_{\gamma^t(Q)} \cup \partial$.  

Let $\b \in \Sigma(Q)$ for which $x \notin R_\b$.  Then there exists, by the compactness of $R_\b$ (a closed set in $M$ compact), some positive minimum distance between $R_\b$ and $x$.  As $|\Sigma(Q)| < \infty$, there exists some positive minimum distance $\Delta$ between any such $R_\b$ and $x$.

Now, $y_n \in R_{\a^n}$ for some $\a^n \in \Sigma(Q)$, and $\a^n \neq \gamma^t(Q)$ for any $t$.  If $x \notin R_{\a^n}$, then the distance between $y_n$ and $x$ is at least $\Delta$.  For $n$ large, this is a contradiction.  Hence, for some $\a := \a^n$, $x \in R_\a$; consequently, there exists at least one valid completion $\a\delta_0 \cdots \in \Sigma(\infty)$, which is a representation of $x$.  It is different from $\gamma^1, \cdots, \gamma^r$, a contradiction.  \end{proof}

\subsection{Details of the New Proof.}\label{subsecCorrect}  In this subsection, we prove Theorem~\ref{thmfullHD1}.  This proof is also the author's correction of the proof of Theorem~\ref{thmUrbexpanding}.  

Choose a Markov partition with small diameter $\Rc := \{R_1, \cdots, R_s\}$.  It is easy to see that the number of representations of every point is less than or equal to some natural number $P_0$.

Before giving the proof, we need some definitions and lemmas.  Two elements of the same generation are \textbf{adjacent} if their intersection is nonempty and \textbf{non-adjacent} if their intersection is empty.  For a point $x \in M$, define the \textbf{adjacency set of $x$ in generation $N$}:  \[\Phi_N(x) = \{R \in G(N) \mid R \ni x\}.\] 

The next two lemmas provide some basic facts about missing preimages.

\begin{lemm}\label{lemmDisOrb}Let $E$ be a set of points whose forward orbits miss an open set $U$.  Then $E$ is also a set of points whose forward orbits miss the open set $T^{-n}(U)$ for any $n \in \NN$.
\end{lemm}
\begin{proof}  Let $x \in E$.  Then $\Or_T^+(x) \cap U = \emptyset$.  If $y \in \Or^+_{T}(x) \cap T^{-n}(U)$, then $T^n(y) \in \Or_T^+(x) \cap U$, a contradiction.  
\end{proof}

\begin{lemm}\label{coroDisOrb}
Let $E$ be a set of points whose forward orbit closures miss a point $y$. Then $E$ is also a set of points whose forward orbit closures miss $T^{-n}(y)$ for any $n \in \NN$.\end{lemm}

\begin{proof}  Let $x \in E$.  Then exists an open neighborhood $U$ around $y$ such that $\Or_T^+(x) \cap U = \emptyset$.  Let $w \in T^{-n}(y)$, and let $V$ be any neighborhood of $w$.  If $z \in \Or^+_{T}(x) \cap V\cap T^{-n}(U)$, then $T^n(z) \in \Or_T^+(x) \cap U$, a contradiction.\end{proof}

Finally, recall that $\overline{A}$ denotes the topological closure of a subset $A \subset M$.  Also, for $x \in M$, $\Or^+_T(x)$ denotes the forward orbit of $x$ under the self-map $T$, and, for $\a \in \Sigma(\infty)$, $\a(n)$ denotes $\a_0 \cdots \a_n$.

\begin{proof}[Proof of Theorem~\ref{thmfullHD1}]

\begin{rema}In this proof, if we let $T$ act on a representation, we are implicitly using the aforementioned semi-conjugacy, as this action denotes left shift.\end{rema} \bigskip

By Lemma~\ref{coroDisOrb}, if any two of the $x_1, \cdots, x_p$ have forward orbits that intersect, we may replace both of these points with a point in the intersection of their forward orbits and still prove the theorem.  Repeat.  Hence, without loss of generality, we may assume that $x_1, \cdots, x_p$ have pairwise disjoint forward orbits.

Let $\bar{\gamma}^1, \cdots, \bar{\gamma}^P$ be all possible representations of $x_1, \cdots, x_p$ (all representations of the same point are included in this list).  Hence, $P \leq p P_0$.  Also, there exists a least generation $\tilde{n}$ such that $|G(\tilde{n})| > P$ and $\delta_T \lambda^{-\tilde{n}} < 1$.

Let us collect these representations thus:  \[\{\bar{\gamma}^1\}, \cup_{t=0}^{3s}\{T^t(\bar{\gamma}^2)\}, \cdots, \cup_{t=0}^{3s}\{T^t(\bar{\gamma}^P)\}.\]  From each collection, pick exactly one element; call this element $\tilde{\gamma}^j$.  Because of the pairwise disjoint orbits, the chosen elements are distinct representations.  Hence, there exists $\tilde{N} \in \NN$ such that for all $n \geq \tilde{N}$, $\tilde{\gamma}^1(n-2s), \cdots, \tilde{\gamma}^P(n-2s)$ are distinct.  Repeat over all such possible combinations, and take the largest $\tilde{N}$.  

If $T^t(\bar{\gamma}^j) = a \cdots$ for some general block $a$, set $Q_{j,t}=8s-4$.  Otherwise, after the first general block $a$, there exists a general block $b$ of least last index $J\geq1$ such that $a \neq b$ (i.e. $a$ is not equivalent to $b$), and set $Q_{j,t} = \max(J+ 2s, 8s-4)$.  Set $Q = \max\{Q_{j,t} \mid j=1, \cdots, P \textrm{ and } t=0, \cdots, 3s\}.$

Let \[q_0 = \max(\tilde{N}, 2sP+1,Q, \tilde{n}).\]  By Lemma~\ref{lemmDistinctDisjoint}, there exists a sequence of integers $\{q_i\}$ greater than or equal to $q_0$ such that $\bar{\gamma}^1_{q_i-1}$ is nondegenerate.  Let us now fix a $q \in \{q_i\}$ and set \[\gamma^1 :=\bar{\gamma}^1_0 \cdots \bar{\gamma}^1_{q}.\]  Thus, Lemma~\ref{lemmPartMatch} applies to every such $\gamma^1$.

For each $2 \leq j \leq P$, there exists, by Lemma~\ref{lemmDistinctDisjoint}, a least $K_j \in \{0, \cdots, s-1\}$ such that $(T^{K_j}(\bar{\gamma}^j))_{q-1}$ is nondegenerate.  Set \[\gamma^j :=(T^{K_j}(\bar{\gamma}^j))_0 \cdots (T^{K_j}(\bar{\gamma}^j))_{q}.\]  Note that Lemma~\ref{lemmPartMatch} can individually apply to each $\gamma^j$.

Define \begin{align} E_k := E_k(q) := \{& R_\a \mid \a \in \Sigma(kq) \textrm{ and } T^n(R_\a) \cap \cup_{j=1}^P\textrm{Int}R_{\gamma^j} = \emptyset \nonumber \\  & \textrm{for every } n=0, 1, \cdots, (k-1)q\}.\nonumber\end{align}  Hence, $R_\a \in E_k$ if and only if all of the $\gamma^1, \cdots, \gamma^P$ are not substrings of $\a
$.
 
As in~\cite{Ur}, we wish to show that $\{E_k\}$ is strongly tree-like so that we can apply Lemma~\ref{lemmMMUr}.  Let \[K := \cup E_1\] and \[E(q) := \cap_{k=1}^\infty \cup E_k.\]  As in~\cite{Ur}, we note that \[d_k < \delta_T \lambda^{-qk}\] by (\ref{MP13}).  Verifying (\ref{MM2}), (\ref{MM3}), (\ref{MM5}), and (\ref{MM6}) is routine and can be found in~\cite{Ur}.    

We, however, must correct the estimate of $\Delta_k$.  (With this estimate, we will also verify (\ref{MM4}).)  Let $R_\a \in E_k$.  Thus, no match of any of the $\gamma^j$'s exists with $\a$.  If, for a $\gamma^j$, there are no partial matches, then any valid concatenation (of the correct length) of $\a$ will produce an element of $E_{k+1}$.  For the remaining $\gamma^j$'s, there are partial matches, and we will apply Lemma~\ref{lemmPartMatch} serially.  Each of these remaining $\gamma^j$'s has a partial match with smallest head.  Pick one of these $\gamma^j$'s (needs not be unique) with the least smallest head $h$; call it $\gamma$.  Let $\gamma'$ be one of the $\gamma^j$'s except for $\gamma$, and denote the smallest head of $\gamma'$ by $h'$.  Thus $h \leq h'$.

Let $b^0$ and $b^1$ be chosen as in Lemma~\ref{lemmPartMatch} applied to $\gamma$.  Then there is no match of $\gamma$ with $\a b^0 b^1 \b_0\cdots \b_{k'}$ where $\b_0, \cdots, \b_{k'}$ are any letters that make $\a b^0 b^1 \b_0\cdots \b_{k'} \in \Sigma((k+1)q)$. 

\begin{sublemm}There is no match of $\gamma'$ with $\a b^0 b^1$.
\end{sublemm}

\begin{proof}  Assume not.  Let us denote $\a = \a_0 \cdots \a_N$.  Since $l(b^0 b^1) \leq 2s$, all but at most the last $2s$ letters of $\gamma'$ are in the partial match with head $h'$.  Consequently, all but at most the last $2s$ letters of $\gamma$ are, likewise, in the partial match with head $h$.  It is easy to see that \[N + 2s - h' \geq q,\] and therefore \[h' -h \leq 2s\] since $h \geq N-q$.  Now, by construction, $T^{h'-h}(\gamma)$ and $\gamma'$ disagree on at least the last $2s+1$ letters, a contradiction as both are partial matches with $\a$. \end{proof}

Remove $\gamma$ from consideration.  Now pick, among the remaining, one with the least smallest head (again, needs not be unique), and repeat applying Lemma~\ref{lemmPartMatch} with $\a$ replaced by $\a b^0 b^1$ until no more $\gamma^j$'s remain (possible since $q > 2sP$).  Therefore, after serially applying Lemma~\ref{lemmPartMatch}, we obtain \[R_{\a b^0 b^1 \cdots b^{2P-1} b^{2P} \b_0 \cdots \b_{k'}} \in E_{k+1}\] where $\b_0, \cdots \b_{k'}$ are any allowed letters.  Thus, $R_{\a b^0 b^1 \cdots b^{2P-1} b^{2P}}$ is a union of elements of $E_{k+1}$.

By Lemma~\ref{lemmBDII} and the monotonicity of $\varepsilon(\cdot)$, \[\varepsilon(2sP) \leq C \frac{\sigma(R_{\a b^0 b^1\cdots b^{2P-1} b^{2P}})}{\sigma(R_{\a})}.\]

Consequently, $\sigma(\cup E_{k+1} \cap R_\a) \geq \frac{\varepsilon(2sP)}{C} \sigma(R_\a)$.  Thus, $$\textrm{density}(E_{k+1}, R_\a) \geq \frac {\varepsilon(2sP)} C.$$

Hence, $$\frac {\sum_{j=1}^k \log \Delta_j} {\log d_k} \leq \frac{k \log \frac {\varepsilon(2sP)} C}{\log \delta_T -qk \log \lambda}.$$

Applying Lemma~\ref{lemmMMUr}, we obtain  $$HD(E(q)) \geq \dim M + \frac {\log\frac {\varepsilon(2sP)} C}{q\log\lambda} \geq \dim M + \frac {\log\frac {\varepsilon(2spP_0)} C}{q\log\lambda}$$ for all $q \in \{q_i\}$. 

\begin{sublemm}
The set $E(q)$ is also a set of points whose forward orbits miss neighborhoods of $x_1, \cdots, x_p$.
\end{sublemm}

\begin{proof}  For interior points in $\{x_1, \cdots, x_p\}$, apply Lemma~\ref{lemmDisOrb}.

Let $x \in \{x_1, \cdots, x_p\}$ be a boundary point.  By (the proof of) Lemma~\ref{lemmBndPoints}, there exists an open set $U \ni x$ such that $U \subset \cup \Phi_{q+s}(x)$.  We claim that all the points in $E(q)$ have forward orbits which miss the open set $U$.  Assume not.  Then there exist a $y \in E(q)$, which corresponds to an $\a \in \Sigma(\infty)$, and $n$ such that $T^n(y) \in U$.  Let $k \in \NN$ be chosen so that $kq \geq n+q+s$. Hence, $T^n(R_{\a(kq)}) \cap U \neq \emptyset$.  Thus, $T^n(R_{\a(kq)}) \subset R_{\b}$ for some $\b \in \Phi_{q+s}(x)$.  Now $\b$ is equivalent to one of the representations of $x$; say it is $\bar{\gamma}^j$.  Thus, $T^{n+K_j}(R_{\a(kq)}) \subset T^{K_j}(R_{\b}) \subset T^{K_j}(R_{\bar{\gamma}^j(q+K_j)}) = R_{\gamma^j}$, a contradiction. \end{proof}

Letting $q_i \rightarrow \infty$, we have shown our desired result:  for any points $x_1, \cdots, x_p \in M$, \[F_T(x_1, \cdots, x_p) := \{x \in M \mid \{x_1, \cdots, x_p\} \cap \overline{\Or^+_T(x)} = \emptyset\}\] has full Hausdorff dimension (i.e. = $\dim M$). \end{proof}

\begin{rema} If one simply wishes to correct the proof of Theorem~\ref{thmUrbexpanding}, one can significantly simplify the above proof by considering missing only one point, an interior point.  
\end{rema}

In the next section, we will see how the author's correction leads to a generalization in dimension one.

\section{A Generalization}  In this section, we prove Theorem~\ref{thmWin1} (or, more precisely, Theorem~\ref{thmWin2} below).  From this theorem, we obtain a useful corollary.  First, however, we must continue our study of Markov partitions from Subsection~\ref{subsecMoreMP}.

\subsection{Even More Markov Partitions}\label{subsecEvenMoreMP} In this subsection, we make a second refinement of the bounded distortion property and further study boundary points.  As in Subsections~\ref{subsecMPBasics} and~\ref{subsecMoreMP}, let us consider a Markov partition with small diameter $\Rc := \{R_1, \cdots, R_s\}$ for $T$.  

Our second refinement of bounded distortion (Theorem~\ref{thmBDI}) is as follows.  Let $R_{\min} \in \Rc$ be an element with smallest $\sigma$, and let $R_{\max} \in \Rc$ be an element with largest $\sigma$.  Define $r = \frac {\sigma(R_{\min})} {\sigma(R_{\max})}.$

\begin{lemm}\label{lemmBDIII}  Let $N \in \NN$ and $\eta$ be a valid finite string of length at least 2.  Let $R_{\eta \a}$ be an element of $G(N)$ (contained in $R_\eta$) of largest $\sigma$; let $R_{\eta \b}$ be an element of $G(N)$ of smallest $\sigma$. Then \[\frac {\sigma(R_{\eta \b})}{\sigma(R_{\eta \a})} \geq \frac r C.\]
\end{lemm}

\noindent (Note that $C$ is from Theorem~\ref{thmBDI}.)

\begin{proof} Let $\eta$ have length $n$.  Consider \[ r \leq \frac{\sigma(R_{\b_{N-n}})} {\sigma(R_{\a_{N-n}})} = \frac{\int_{R_{\eta\b}} J(T^{N})(x) d\sigma(x)}{\int_{R_{\eta \a}} J(T^{N})(x) d\sigma(x)} \] \[\leq \frac{\sigma(R_{\eta\b}) \sup\{J(T^{N})(x) \mid x \in R_{\eta}\}}{\sigma(R_{\eta \a}) \inf\{J(T^{N})(x) \mid x \in R_{\eta}\}} \leq C \frac{\sigma(R_{\eta\b})}{\sigma(R_{\eta\a})}.\] \end{proof}

Let us now further study boundary points.  

\begin{lemm}  \label{lemmBndToBnd}The following hold (for Markov partitions with small diameter):  \[T(\partial) \subset \partial \textrm{ and } T^{-1}(\partial) \subset \partial .\]\end{lemm}

\begin{proof} Let $x \in \partial$; thus, $x$ is on the boundary of $R_\a$ for $\a \in \Sigma(n)$ for some $n \in \NN \cup \{0\}$.  Since the boundary of all elements of any given generation is contained in the boundary of all elements of the next generation, we can choose $n$ to be as large as we like.  By (\ref{MP13}), choose $n$ so large that the diameter of every element of $G(n)$ is $< \delta_T/2$.  By the proof of Lemma~\ref{lemmReps}, there exists $\b \neq \a \in \Sigma(n)$ such that $x$ is on the boundary of $R_\b$.  Thus, diam$(R_\a \cup R_\b) < \delta_T$.

By (\ref{MP9}), $T(R_\a)$ and $T(R_\b)$ are both elements of the previous generation.  They are distinct elements because $T$ is injective on $R_\a \cup R_\b$.  Thus, $T(x)$ lies in two distinct elements of the same generation, and hence it must lie on the boundary.

The map $T$ is an $N$-fold covering for some $N >1$~\cite{KS}.  Thus, the set $T^{-1}(x)$ has $N$ elements.  Let $\{i_1, \cdots, i_m\}$ be the set of letters such that $A_{i_j\a_0}=1$.  By (\ref{MP10}), if $m < N$, there exist elements $y \neq z \in T^{-1}(x)$ and a letter $i:=i_j$ for which $y, z \in R_{i\a}$.  Consequently, $T$ is not injective on $R_{i\a}$, a contradiction.  This also shows that each element of $T^{-1}(x)$ belongs to only one $R_{i_j\a}$.  Also, by (\ref{MP9}), $m \leq N$; thus, $m=N$, and each $R_{i_j\a}$ contains exactly one element of $T^{-1}(x)$.  Likewise for $\b$.  Hence, each element of $T^{-1}(x)$ lies in two distinct (because $\a \neq \b$) elements of the next generation, and thus it must lie on the boundary. \end{proof}

Let us further distinguish subsets of $\partial$.  Let $\partial_n$ denote \textbf{the set of all boundary points of all elements of $G(n)$}.  Clearly, a chain of inclusions $\partial_0 \subset \partial_1 \subset \cdots$ exists.  A point in $\partial_0$ has \textbf{weight} 0.  For $n \geq 1$, a point in $\partial_n \backslash \partial_{n-1}$ has \textbf{weight} $n$.

Also, given $\gamma \in \Sigma(n)$, let us define the following sets of valid concatenations of $\gamma$: \[\Sigma_\gamma(q) := \{\delta \in \Sigma(n+q) \mid \delta \textrm{ is equivalent to } \gamma\}.\]

These notions will be used in Subsection~\ref{sec1DimProof}.

\subsection{The Proof of the Generalization in Dimension One}\label{sec1DimProof}  In this subsection, we prove in dimension one a generalization of Theorems~\ref{thmUrbexpanding} and~\ref{thmfullHD1} (and also the aforementioned result in~\cite{AN}) and, in part, a generalization of Theorem~\ref{thmDani}.  In particular, we prove Theorem~\ref{thmWin1} (or, more precisely, Theorem~\ref{thmWin2} below).  An immediate corollary is also obtained.  

We will, in this subsection, specialize to the one-dimensional case:  consider the one-dimensional system $(S^1, \sigma, T)$ where $S^1 := \RR/\ZZ$, $\sigma$ is the probability Haar measure on $S^1$, and $$T:S^1 \rightarrow S^1$$ is a $C^2$-expanding map.

It is clear from Krzy\.zewski and Szlenk's construction of a Markov partition with small diameter (\cite{KS}, proof of Lemma 4) that

\begin{lemm} \label{lemmConnMPCircle} For any $C^2$-expanding map $T:S^1 \rightarrow S^1$, there exists a Markov partition with small diameter for which every element of every generation is path-connected.
\end{lemm}

Endow $S^1$ with the usual metric, and let $d(A)$ denote the diameter of a set $A$.  Using Lemma~\ref{lemmConnMPCircle}, we obtain a Markov partition with small diameter $\Rc := \{R_1, \cdots, R_s\}$, which we fix. Since the elements of each generation are intervals, we may use $d$ and $\sigma$ interchangeably on these elements.

Recall the definition of $\varepsilon(\cdot)$ from Subsection~\ref{subsecMoreMP}.

\begin{lemm}\label{lemmFit} Let $\Rc := \{R_1, \cdots, R_s\}$ be a Markov partition with small diameter for which every element of every generation is path-connected.  For any closed interval $B$ such that \[d(B) < \min\{d(R_1), \cdots, d(R_s)\},\] there exists $N \in \NN$ for which an element $R_\eta \in G(N-1)$ can be chosen to satisfy \begin{equation} \label{eqnFit}2 d(R_\eta) \geq d(B) \geq \frac{\varepsilon(1)} C d(R_\eta).\end{equation}  Moreover, an element of $G(N)$ lies in both $B$ and $R_\eta$ and at least half of the interval $B$ lies in $R_\eta$.  Finally, if any element of any generation $R_\a \supset B$, then $R_\a \supset R_\eta$.

\end{lemm}

\begin{rema}  Although more than one value of $N$ may make (\ref{eqnFit}) true, we always agree to take the value of $N$ as in the proof below.  Hence,  for each $B$ there exists a unique $N$, namely $G(N)$ is the least generation in which an element of that generation lies completely in $B$.
\end{rema}

\begin{proof}

\medskip\noindent\textbf{Case 1:  $B \cap \partial_0 \neq \emptyset$.}\medskip

By length, $B$ contains exactly one point $y$ of weight $0$.  Thus, we have closed intervals $B^+$ and $B^-$ such that \[B = B^+ \cup B^-\] where \[\{y\} = B^+ \cap B^-.\]  Let $d(B^+) \geq d(B^-)$.  (Note that $B^-$ could possibly be just $\{y\}$.) 

Now there exists a least $N \in \NN$ such that $(\partial_N \backslash \{y\}) \cap B^+ \neq \emptyset$.  Hence, there exists $R_\eta \in G(N-1)$ such that $B^+ \subset R_\eta$.  Thus, \[d(B) \leq 2d(R_\eta).\]

Let $z \in (\partial_N \backslash \{y\}) \cap B^+$ be closest to $y$.  Then the interval between $y$ and $z$ in $B^+$ is an element of $G(N)$.  Denote it by $R_{\eta i}$.  Hence, by Lemma~\ref{lemmBDII}, \[\frac {\varepsilon(1)} C d(R_\eta) \leq d(B).\]  

\medskip\noindent\textbf{Case 2: $B \cap \partial_0 = \emptyset$.}\medskip

Thus, there exists a least $N \in \NN$ such that $\partial_{N} \cap B \neq \emptyset$.  

\medskip\noindent\textbf{Case 2A: $|\partial_{N} \cap B| \geq 2$.}\medskip

Thus, there exists $R_\eta \in G(N-1)$ such that $B \subset R_\eta$.  Moreover, there exists an element $R_{\eta i}$ such that $R_{\eta i} \subset B$.  As in Case 1, we obtain (\ref{eqnFit}).

\medskip\noindent\textbf{Case 2B: $|\partial_{N} \cap B| = 1$.}\medskip

Let $y$ be the point of weight $N$ in $B$.  Repeat the proof of Case 1 with this $y$. \end{proof}

To prove our generalization, we require more notation.  Let us quote some of Schmidt's original notation from~\cite{Sch2}.  We play Schmidt's game on a complete metric space $\tilde{M}$.  Let $0 < \kappa < 1$.  Given a ball $B$ of $\tilde{M}$ with radius $\tilde{r}$, let $B^\kappa$ denote the set of all balls $B' \subset B$ with radius equal to $\kappa \tilde{r}$.  

Also, recall that we denote the $(0,Q)$-substring of a $\gamma \in \Sigma(\infty)$ by $\gamma(Q)$.  Finally, note that $C$ is from Theorem~\ref{thmBDI}.  Our generalization is

\begin{theo}\label{thmWin2}Let $x_0 \in S^1$.  Then \[F_T(x_0) := \{x \in S^1 \mid x_0 \notin \overline{\Or_{T}^+(x)}\}\] is an $\frac {\varepsilon(7s+2)}{2C}$-winning set.  (If $\Rc$ has no degenerate letters, we may replace $\varepsilon(7s+2)$ with $\varepsilon(5)$.)
\end{theo}

\begin{proof} Let $M := S^1$ and $F := F_T(x_0)$.  Let $\gamma \in \Sigma(\infty)$ be a representation of $x_0$.  

Let $n := \frac {\varepsilon(7s+2)}{2C}$ and $0 < m < 1$.  We show that $F$ is $(n,m)$-winning.  Black starts, choosing $B_1$.  Now there is a least $J \in \NN$ such that for any choice of $B_J$, \[d(B_J) < \min_{\xi \in \Sigma(1)} (d(R_\xi)).\]  (White chooses any allowed sets for $W_1, \cdots, W_{J-1}$.  Black chooses $B_J$.)  

By Lemma~\ref{lemmFit}, there exist $N_0 \geq1$ and an element $R_\eta \in G(N_0)$ that contains at least half of $B_J$.  Since $n \leq 1/2$, choose $W_J \subset R_\eta$.  

Let us now refine the notion of constants of bounded distortion:

$$\varepsilon_\eta(q) := \min_{\delta \in \Sigma_\eta(q)} \frac{\sigma(R_\delta)} {\sigma(R_\eta)} > 0.$$

$$1 \geq \E_\eta(q) := \max_{\delta \in \Sigma_\eta(q)} \frac{\sigma(R_\delta)} {\sigma(R_{\eta})} > 0.$$

For the given $\eta$, Lemma~\ref{lemmBDIII} implies that $\frac{\varepsilon_\eta(q)}{\E_\eta(q)} \geq r/C.$  

\begin{sublemm} $\E_\eta(q) \geq s^{-q}$.
\end{sublemm}

\begin{proof}  There are at most $s^q$ elements of $G(l(\eta)-1+q)$ which are contained in $R_\eta$, i.e. $|\Sigma_\eta(q)| \leq s^q$, because there are only $s$ possible letters to append (on the right) to any finite string.

Let $R_\a \in G(l(\eta)-1+q)$ be such that $\E_\eta(q) = \frac {\sigma(R_\a)} {\sigma(R_\eta)}$.  Then $R_\a$ has the largest $\sigma$ of any element of $G(l(\eta)-1 +q)$ contained in $R_\eta$.  Because all elements of the same generation have pairwise disjoint interiors and $\partial$ is $\sigma$-null, $s^q \sigma(R_\a) \geq \sum_{\b \in \Sigma_\eta(q)} \sigma(R_\b) = \sigma(R_\eta)$.
\end{proof}

Hence, $\varepsilon_\eta(q) \geq \frac{r}{C s^q}$.

Define $$H_k = H_k(Q) = \{R_\a \mid \a \in \Sigma(Q + k) \textrm{ and }  T^n(R_\a) \cap \textrm{Int}R_{\gamma(Q)} = \emptyset $$ $$\textrm{ for every } n = 0, 1, \cdots, k\}.$$

There exists a least $P \in \NN$ such that  \begin{enumerate} \item $P \geq 4s-2$ and \item $\frac{4 C^4 \delta_T \lambda^{-P}}{\varepsilon(1) \varepsilon(2s) \varepsilon(7s+2)r d(R_{\max})} < \frac{r \varepsilon(1)} {2C^2}.$ \end{enumerate} Also, there exists a least $L_0 \in \NN$ such that $s^{-1/L_0} \geq \lambda^{-1/2}.$

\begin{sublemm}\label{sublemmExpand} For every $q \in \NN$, there exists a least $p \in \NN$ such that any allowed choice of $B_{J+p}$ is a subset of $R_\delta$ for some $\delta \in \Sigma_\eta(q)$. 
\end{sublemm}

\begin{proof}  Recall the definition of $B^+$ from the proof of Lemma~\ref{lemmFit}.  

Note that, by Lemma~\ref{lemmFit}, $p \geq 1$.  It suffices to show the sublemma for some $p$; that a least such $p$ exists is then immediate.  Let $\b$ be an element of $\Sigma_\eta(q)$ with smallest $\sigma$.  Then $\frac {d(R_\b)}{d(R_\eta)} \geq \frac{r}{C s^q}$.  Thus, there exists a large integer $t$ such that $\frac{r}{C s^q} d(R_\eta) > B_{J+t}.$  Hence, $|B_{J+t} \cap \partial_{G_\b}| \leq 1$ (i.e. there is at most one boundary point of the proper weight in $B_{J+t}$). Pick $W_{J+t} \subset B^+_{J+t}$.  Hence, let $p=t+1$.\end{proof}

By Sublemma~\ref{sublemmExpand}, there exists a $L_1 \in \NN$ such that $B_{J+L_1}$ is contained in an element of $G(2P)$.  Let $L:= \max(L_0,L_1)$.

By Lemma~\ref{lemmFit}, there exists a least $N \in \NN$ for which we can choose an element $R_\delta \in G(N-1)$ such that \begin{eqnarray}\label{eqnA}2d(R_\delta) \geq d(B_{J+L}) \geq \frac{\varepsilon(1)} C d(R_\delta).\end{eqnarray}   Also, there exists $R_{\delta k} \subset B_{J+L}$ for some letter $k$.  By construction, $R_{\delta k} \subset R_\eta$, and hence $R_\delta \subset R_\eta$ (because the generation that $R_\delta$ belongs to is later than or the same as that of $R_\eta$). 


Also, since $B_{J+L}$ is contained in an element of $G(2P)$, and $B^+_{J+L}$ (see the proof of Lemma~\ref{lemmFit} for the meaning of the notation) is contained in an element of $G(N-1)$, $N-1 \geq 2P$.

Pick an integer $Q >N$ as follows.   Choose integers $N_4 > N_3 > N_2 > N_1 \geq s+1$ as follows:  $\gamma_{N+N_1}$ is the next nondegenerate letter in $\gamma$ following $\gamma(N+s)$, $\gamma_{N+N_2} $ is the next nondegenerate letter in $\gamma$ following $\gamma(N+N_1)$, $\gamma_{N+N_3}$ is the next nondegenerate letter in $\gamma$ following $\gamma(N+N_2)$, and $\gamma_{N+N_4}$ is the next nondegenerate letter in $\gamma$ following $\gamma(N+N_3)$.   (By Corollary~\ref{corMaxBlock}, $4 + s \leq N_4\leq5s$.)  If $\gamma(N+N_4+1)$ is of the form $a \cdots ab$ for a general block $a$ and $b$ is either a general block not equivalent to $aa$ or a double general block not equivalent to $aa$, then $Q = N+N_1+1$; otherwise, choose $Q = N+N_4+1$. Hence, Lemma~\ref{lemmPartMatch} applies to $\gamma(Q)$.

Now, by (\ref{eqnA}), \begin{align*}(mn)^L d(B_J) & \geq \frac{\varepsilon(1)} C d(R_\delta) \geq \frac{\varepsilon(1)} C \varepsilon_\eta(N-1-N_0) d(R_{\eta}) \\ & \geq \frac{\varepsilon(1)} C \varepsilon_\eta(N-1-N_0) d(B_J)/2 \geq \frac{\varepsilon(1)} {2C} \varepsilon_\eta(Q) d(B_J) \\ & \geq \frac{r \varepsilon(1)} {2C^2} \E_\eta(Q) d(B_J) \geq \frac{r \varepsilon(1)} {2C^2} s^{-Q} d(B_J).\end{align*}

Thus, \begin{eqnarray}\label{eqnB}m \geq \frac{r \varepsilon(1)} {2C^2} \lambda^{-Q/2}.\end{eqnarray}

Since $Q-N \geq s+1$, $R_{\delta k}$ splits into at least two elements of $G(Q)$ by Corollary~\ref{corMaxBlock}.  One of these is not $R_{\gamma(Q)}$; call this element $R_\a$.  (Note that $Q \geq 8s-4$.)  By Lemma~\ref{lemmPartMatch}, there exist strings $b^0$ and $b^1$, each of length at most $s$, such that for any valid choice of letters $\b_0, \cdots,\b_k$, where $l(b^0) + l(b^1) + k +1 \leq Q$, no match of $\gamma(Q)$ with $\a b^0 b^1 \b_0 \cdots \b_k$ exists.  Thus, \begin{eqnarray}\label{eqnC}d(R_\delta) \geq d(R_{\a b^0 b^1}) \geq \frac{\varepsilon(7s+2)} C d(R_\delta)\end{eqnarray} by Lemma~\ref{lemmBDII}.  Consequently, by (\ref{eqnA}) and (\ref{eqnC}), \[d(R_{\a b^0 b^1}) \geq n d(B_{J+L}) \geq \frac{\varepsilon(7s+2) \varepsilon(1)} {2C^2} d(R_{\a b^0 b^1}).\] 

Since White must choose $W_{J+L} \in B_{J+L}^n$, White picks $W_{J+L} \subset R_{\a b^0 b^1}$.  Black now chooses $B_{J+L+1} \in W_{J+L}^m$; hence, \begin{eqnarray}\label{eqnF} m d(R_\a) \geq d(B_{J+L+1}) \geq \frac{m \varepsilon(7s+2) \varepsilon(1) \varepsilon(2s)} {2C^3} d(R_\a).\end{eqnarray}

By Lemma~\ref{lemmFit} again, there exists $N' \in \NN$ for which we can choose an element $R_{\eta'} \in G(N'-1)$ such that \begin{eqnarray}\label{eqnG}2 d(R_{\eta'}) \geq d(B_{J+L+1}) \geq \frac{\varepsilon(1)} C d(R_{\eta'}).\end{eqnarray}  Also, there exists $R_{\eta' k'} \subset B_{J+L+1}$ for some letter $k'$.  Now, by construction, $R_{\eta'} \subset R_{\a b^0 b^1}$.  Hence, $N'-1 \geq Q+l(b^0) + l(b^1)$.  Define $q_{J+L+1} = N' -Q$.

\begin{sublemm}  $l(b^0) + l(b^1) < q_{J+L+1} \leq Q$.
\end{sublemm}

\begin{proof}  Assume that $q_{J+L+1} \geq Q+1$.  We have \[d(R_{\eta'}) \leq \E(q_{J+L+1}-1) C d(R_\a) \leq \E(Q) C d(R_\a).\]  Let $R_\b \in G(Q)$ such that $\E(Q) = \frac {d(R_\b)}{d(R_{\b_0})}.$  Since $d(R_{\b_0}) \geq r d(R_{\max})$ and (\ref{MP13}) holds, $d(R_{\eta'})  \leq \frac{\delta_T \lambda^{-Q}}{r d(R_{\max})} C d(R_\a)$.

Hence, by (\ref{eqnF}) and (\ref{eqnG}), \[m \leq \frac{4 C^4 \delta_T \lambda^{-Q}}{\varepsilon(1) \varepsilon(2s) \varepsilon(7s+2)r d(R_{\max})} \]\[\leq \frac{4 C^4 \delta_T \lambda^{-P}}{\varepsilon(1) \varepsilon(2s) \varepsilon(7s+2)r d(R_{\max})} \lambda^{-Q/2} < \frac{r \varepsilon(1)} {2C^2} \lambda^{-Q/2},\] a contradiction of (\ref{eqnB}).
\end{proof}

Consequently, by Lemma~\ref{lemmPartMatch}, no match of $\gamma(Q)$ with any valid string beginning with $\a b^0 b^1$ in $\Sigma(Q+q_{J+L+1})$ exists.

Now, by construction, $B_{J+L+1}$ contains an element (i.e. $R_{\eta'k'}$ of $G(Q + q_{J+L+1})$) whose string begins with $\a b^0 b^1$.  Let $\a':= \eta'k'$.  Thus, $$R_{\a'} \in H_{q_{J+L+1}}.$$  By Lemma~\ref{lemmPartMatch}, there exist strings $b'^0$ and $b'^1$, each of length at most $s$, such that for any valid choice of letters $\b'_0, \cdots,\b'_k$, where $l(b'^0) + l(b'^1) + k +1 \leq Q$, no match of $\gamma(Q)$ with $\a' b'^0 b'^1 \b'_0 \cdots \b'_k$ exists.  Thus, \[d(R_{\eta'}) \geq d(R_{\a' b'^0 b'^1}) \geq \frac{\varepsilon(7s+2)} C d(R_{\eta'})\] by Lemma~\ref{lemmBDII}.  

As before, White chooses $W_{J+L+1} \subset R_{\a' b'^0 b'^1}$.  Continue thus by induction.

Therefore, we obtain \begin{equation}\label{eqnD} \cap_{p=J+L+1}^\infty W_p \in \cap_{p=J+L+1}^\infty (\cup H_{\sum_{j=J+L+1}^p q_j}(Q)).\end{equation}  The latter set is a set of points whose forward orbits avoid Int$R_{\gamma(Q)}$.  

Denote $$A_\gamma := \cup_{Q=2P+2}^\infty \cap_{p=J+L+1}^\infty (\cup H_{\sum_{j=J+L+1}^p q_j}(Q)).$$   By (\ref{eqnD}), $A_\gamma$ is $(n,m)$-winning for all $0 < m <1$.

If $\gamma$ is the unique representation of $x_0$, then, by Lemma~\ref{lemmReps}, $x_0 \in$ Int$R_{\gamma(Q)}$ for all $Q \in \NN \cup \{0\}$.  Hence, $A_\gamma$ is the set of points whose forward orbits avoid a neighborhood of $x_0$.  Thus, we are done for $x_0$ in this case.

If $\gamma^1, \cdots, \gamma^{r_0}$ are representations of $x_0$ for $r_0 >1$, then $A:=\cap_{t=1}^{r_0} A_{\gamma^t}$ is $n$-winning.  The set of boundary points is the countable union of finite sets and hence countable (for $M=S^1$).  Thus, $A \backslash \partial$ is $n$-winning.

Let $x \in A \backslash \partial$.  Then there exist some $Q_1, \cdots, Q_{r_0}$ such that $$\Or^+_T(x) \cap\textrm{Int}R_{\gamma^t(Q_t)} = \emptyset.$$  Let $Q := \max(Q^t)$.  By Lemma~\ref{lemmBndPoints}, there exists an open neighborhood $U$ of $x_0$ such that $U \subset \cup_{t=1}^{r_0} \textrm{Int}R_{\gamma^t(Q)} \cup \partial$.

If there exists $q \geq 0$ such that $T^q(x) \in \partial$, then, by Lemma~\ref{lemmBndToBnd}, $x \in \partial$, a contradiction.  Thus, $\Or^+_T(x) \cap \partial = \emptyset.$  Hence, $\Or^+_T(x) \cap U = \emptyset.$  Thus, $A \backslash \partial$ is a set of points whose forward orbits avoid an open neighborhood of $x_0$.  \end{proof}

We have the following corollary.  Let $\{T_n\}_{n=1}^N$ be any finite set of $C^2$-expanding self-maps of $S^1$.  For each map, choose, via Lemma~\ref{lemmConnMPCircle}, a Markov partition with small diameter with only intervals as elements.  Let $s_n$ be the number of elements of the $n^{th}$ Markov partition.  Let $\varepsilon_n$ be the lower constant of bounded distortion for the $n^{th}$ Markov partition. Let $C_n$ be the constant (from Theorem~\ref{thmBDI}) for the $n^{th}$ Markov partition.  Let $\a = \min(\frac {\varepsilon_1(7s_1+2)}{2C_1}, \cdots, \frac {\varepsilon_N(7s_N+2)}{2C_N}) > 0$.

\begin{coro}  \label{cor1DimWin} For each $n$, choose a (at most) countably infinite set $\{x^n_i\}_{i=1}^{\infty} \subset S^1$.  Then \begin{eqnarray} \label{eqnE} \bigcap_{n=1}^N \bigcap_{i=1}^\infty F_{T_n}(x^n_i)\end{eqnarray} is $\a$-winning.
\end{coro}

\begin{ques}\label{quesonehalfwin}  Is $F_T(x_0)$ $\a$-winning for some $\a$ independent of the choice of Markov partition and of $T$ itself (such as $\a = 1/2$ for example)?
\end{ques}

\section{Conclusion}\label{secConclude}  In this note, we have presented a way of proving Theorem~\ref{thmfullHD1} using elementary methods of Markov partitions.  As mentioned, A. G. Abercrombie and R. Nair have another method using higher dimensional nets and Kolmogorov's consistency theorem~\cite{AN}.  In addition to our result, their method also gives a lower bound for the Hausdorff dimension of the set of points whose forward orbits miss balls (of a radius which one can choose, subject to certain constraints) around the points $x_1, \cdots, x_p$.  Instead of constructing good strings as we do, they construct a certain Borel measure on the set of points whose forward orbits miss the desired balls.  This measure encapsulates the iterations of $T$ and is zero on the strings which come too close to hitting the balls to be avoided.  Thus, they are freed from considering matching.  

Our method, on the other hand, is concerned with matching.  In particular, the use of the No Matching lemma requires manipulation and coordination of elements of certain generations of the Markov partition, which the author only knows how to do when the points being missed are contained in these elements.  If one would like to show a result concerning missing balls around points, then one must be able to manipulate and coordinate elements adjacent to the elements which contain the points being missed.  This requirement is most clearly seen when one wishes to miss an interior point, as how close the point is to the boundary of the element (of the requisite generation of the Markov partition) determines how large a ball around this point our method allows us to miss.  This sort of variation does not seem to allow us to give, without further modifications to our method, a lower bound like Abercrombie and Nair's.

However, our elementary method is very geometric since we handle elements of generations of the Markov partition directly.  It is this geometric nature that allows us to generalize, in dimension one, Theorem~\ref{thmfullHD1} and Abercrombie and Nair's result to winning sets.  Doing so has allowed us to obtain a considerable strengthening:  the countable intersection property.  With this property, we can generalize to finitely many maps and countably many points, as precisely stated in Corollary~\ref{cor1DimWin}.  (If we can answer Question~\ref{quesonehalfwin} affirmatively, then we can generalize to countably many maps.) Can we also generalize to winning sets for higher dimensional manifolds, and can we prove a similar result for Anosov diffeomorphisms?  Only starting with Subsection~\ref{sec1DimProof} did we specialize to dimension one.  Much of the theory works for higher dimensions.  How much will work and with what modifications?

\section*{Acknowledgements} The author would like to thank his advisor, Professor Dmitry Kleinbock, for helpful discussions.  The author would also like to thank Professor Mariusz Urba\'nski for his helpful comments.

\end{document}